\documentclass[a4paper]{article}
\usepackage{amsthm,amssymb,amsmath,enumerate,graphicx,psfrag}
\usepackage{algorithm}

\newtheorem{definition}{Definition}

\newtheorem{theorem}[definition]{Theorem}

\newtheorem{lemma}[definition]{Lemma}

\newcommand{\comment}[1]{}

\newcommand{\emtext}[1]{\text{\em #1}}

\newcommand{\mymargin}[1]{%
  \marginpar{%
    \begin{minipage}{\marginparwidth}\small%
      \begin{flushleft}%
        #1%
      \end{flushleft}%
   \end{minipage}%
  }%
}%

\renewcommand{\mymargin}[1]{}%

\newcommand{\sm}{\setminus}
\newcommand{\mm}{-} 
\newcommand{\selfcite}[1]{$\!\!${\bf\cite{#1}}}

\newcommand{\cf}{\kappa}

\newcommand{\mfc}{\ensuremath{M_{\rm FC}}}
\newcommand{\mc}{\ensuremath{M_{\rm C}}}
\newcommand{\mfb}{\ensuremath{M_{\rm FB}}}

\newcommand{\I}{\ensuremath\mathcal{I}}
\newcommand{\Imax}{\ensuremath\mathcal{I^{\rm max}}}

\newcommand{\eends}{\ensuremath\mathcal{E}} 
\newcommand{\itop}[1]{||#1||}

\title{Matroid and Tutte-connectivity  in infinite graphs}
\author{Henning Bruhn}
\date{}

\begin{document}
\maketitle

\begin{abstract}
We relate matroid connectivity to Tutte-connectivity in an infinite graph. 
Moreover, we show that the two cycle matroids, the finite-cycle matroid
and the cycle matroid, in which also infinite cycles are taken into account, 
have the same connectivity function. As an application we re-prove
that, also for infinite graphs, Tutte-connectivity is invariant under 
taking dual graphs.
\end{abstract}

\section{Introduction}

This work is part of a project to develop a theory for infinite
matroids that is analogous to its finite counterpart. 
In the initial paper of this project~\cite{infaxioms}, 
we extended extended previous work of Higgs~\cite{Higgs69,Higgs69c} 
and Oxley~\cite{Oxley92}  by giving equivalent 
 definitions of (finite or infinite)
matroids in terms of independence, bases, circuits, closure and (relative) rank,
just as one is used to for finite matroids. 
Since then, in a series of papers~\cite{2sepmat,infintmat,matunion,matintbase,thindual}, 
several other aspects of infinite matroids have
been explored, among them graphic matroids~\cite{matgraphs}
and  matroid connectivity~\cite{linkingthm}.

These two last aspects are the focus of the current work: Connectivity in 
graphic matroids. For cycle matroids of finite graphs matroid connectivity
translates into a purely graph theoretic notion.
A graph  $G$ is \emph{$k$-Tutte-connected} if for every $\ell\leq k$ and every
partition $X,Y$ of its edge set into sets of at least $\ell$ edges each, 
the number of vertices incident with both an edge in $X$ and an edge in $Y$
is greater than~$\ell$. 
Tutte~\cite{tutte66} proved that a finite graph is $k$-Tutte-connected if and only if
its cycle matroid is $k$-connected.

The main result of this work is an extension of this fact to infinite graphs 
and matroids. For this, let us call a graph $G$ \emph{finitely separable}
if any two vertices may be separated by the deletion finitely 
many edges, and let us define its \emph{finite-cycle matroid},
by declaring any edge set not containing the edge set of a finite cycle 
to be independent. 

\begin{theorem}\label{mainthm}
Let $k\geq 2$ be an integer.
A finitely separable graph is $k$-Tutte-connected if and only if its finite-cycle matroid
is $k$-connected.
\end{theorem}

If the graph in the theorem is infinite, the finite-cycle matroid clearly will be 
infinite as well. But what does it mean for an infinite matroid to be $k$-connected?
A finite matroid $M$ is $k$-connected if for any $\ell\leq k$ and any partition 
of its ground set into two sets $X,Y$ of at least $\ell$ elements each
it follows that
\(
r(X)+r(Y)-r(M)\geq \ell.
\)
Clearly, this definition is useless for infinite matroids as the involved
ranks will usually be infinite. In~\cite{linkingthm} we therefore gave 
a rank-free definition that carries over to infinite matroids. 
To argue that our definition is the right one, we showed that this notion
of connectivity has the same properties as in finite matroids and we,
furthermore, extended Tutte's linking theorem to at least a large subclass
of infinite matroids. Theorem~\ref{mainthm} confirms our claim further. 

\medskip
In~\cite{matgraphs}, we observed that any {finitely separable}
graph has not one but two cycle matroids: The {finite-cycle matroid} 
and the \emph{cycle matroid}, in which any edge set containing 
a finite or infinite cycle is said to be dependent. Here, an infinite cycle 
in the graph is the homeomorphic image of the unit circle in a natural
topological space obtained from the graph (often by compactifying it). This definition
was proposed by Diestel and K\"uhn in a completely graph-theoretical context 
and was subsequently seen to be extremely fruitful as it allows to extend
virtually any result about cycles in a finite graph to at least 
a large class of infinite graphs; see Diestel~\cite{diestelBook10} for an introduction.  

The cycle matroid and the finite-cycle matroid coincide in a finite 
graph but will usually be different in infinite graphs. However, 
as we shall observe in Theorem~\ref{mat:sameconn}, they always have the same connectivity
and even the same connectivity function.

Finally, as an application of our argumentation, we get another 
extension of a result known for finite graphs: Tutte-connectivity 
is invariant under taking duals. 
\begin{theorem}$\!\!${\bf\cite{endduality}}\label{tconnthm}
Let $G$ and $G^*$ be a pair of dual graphs, and let $k\geq 2$.
Then $G$ is $k$-Tutte-connected if and only if $G^*$ is
$k$-Tutte-connected.
\end{theorem}
We remark that this is not a new result. In~\cite{endduality}
we gave a graph-theoretical proof. Here, we will see a matroidal variant.

\section{Infinite cycles}

A graph is \emph{finitely separable} if any two vertices can be separated by
finitely many edges. Let us fix a finitely separable graph $G=(V,E)$ in this section.

A \emph{ray} of $G$ is a one-way infinite path. Two rays are \emph{edge-equivalent} if
for every finite set of edges $F$ there is a component of $G-F$ that contains
subrays of both rays. The equivalence classes of this relation are the 
\emph{edge-ends $\eends(G)$} of~$G$.

We view the edges of $G$ as disjoint homeomorphic images of the unit interval $[0,1]$, 
and define the quotient space $X_G$ by identifying these copies of $[0,1]$
at their common endvertices. Let us define a topological space $\itop G$ on 
$X_G\cup\eends(G)$ by specifying the basic open sets: These are all sets
of the form $\tilde C$, which consists of a topological component of $X_G-Z$
for some finite set $Z$ of inner points of edges together with all edge-ends
that have a ray lying entirely in $C$. We remark that normally this space will 
not be Hausdorff: No edge-end can be separated from a vertex that sends 
infinitely many edge-disjoint paths to one of its rays. However, 
and this is the reason for imposing finite separability, 
two vertices may always be topologically distinguished.
For a locally finite $G$, that is, a graph in which every vertex has finite degree, 
the space $\itop G$ coincides with the Freudenthal compactification.

\medskip

For us a \emph{cycle of $\itop G$} is  
 a homeomorphic image of the unit circle $S^1$ in $\itop G$. This definition of cycles
includes the traditional finite cycles but allows also other cycles, which then 
contain necessarily infinitely many vertices and edges. An \emph{arc} in $\itop G$
is the homeomorphic image of the unit interval $[0,1]$.
A \emph{standard subspace} of $\itop G$ is the closure of a subgraph of $G$ in $\itop G$.
The set of edges that are completely contained in a standard subspace $X$
are denoted by $E(X)$.
Cycles as well as arcs that have their endpoints in $V\cup\eends(G)$ are standard 
subspaces~\cite{Moritz}.  
A \emph{topological spanning tree} of $\itop G$ is a standard subspace that is path-connected in $\itop G$
and which contains every vertex of $G$ but no cycle. For more details see~\cite{matgraphs}.

\begin{figure}[ht]
\centering
\includegraphics[scale=1]{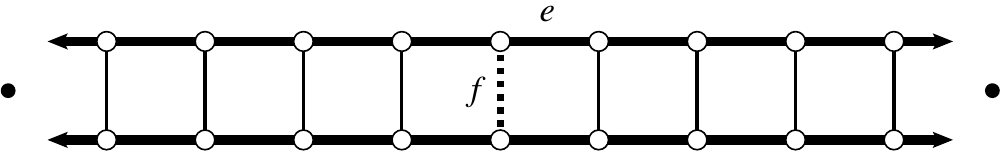}
\caption{An infinite cycle in the double ladder}\label{dladder}
\end{figure}

In Figure~\ref{dladder} some of the introduced concepts are illustrated. The graph there, the double 
ladder, has two edge-ends, one to the left and one to the right. The infinite cycle $C$ in bold lines 
goes through these two edge-ends. 
Moreover, while $C+f$ is a spanning tree of the graph it is not
(even including the two edge-ends) a topological spanning tree, simply because it contains
the infinite cycle $C$. On the other hand, $C-e$ can be seen to be one. Its connectivity
is ensured by the edge-ends. 

\section{Infinite matroids}

As finite matroids,
infinite matroids come with a number of different axiom systems. We only describe here the
\emph{independence axioms}. Let $E$ be a set, let $\I\subseteq 2^E$ be a set of subsets of $E$,
and denote by $\Imax$ the sets in $\I$ that are maximal under inclusion.
We say that $M=(E,\I)$ is a \emph{matroid} with independent sets $\I$ if 
the following axioms are satisfied:
\begin{itemize}
\item[\rm (I1)] $\emptyset\in\mathcal I$.
\item[\rm (I2)] $\I$ is closed under taking subsets, 
that is if $I\in\I$ and $J\subseteq I$ then $J\in \I$.
\item[\rm (I3)] For all $I\in\I\sm\Imax$ and $I'\in\Imax$ there is an $x\in I'\sm I$ such that $I\cup\{x\}\in\I$
\item[\rm (IM)] The set 
$\{\,I'\in\I: I\subseteq I'\subseteq X\,\}$ has a maximal element, 
whenever $I\subseteq X\subseteq E$ and $I\in\I$.
\end{itemize}
Infinite matroids show the same properties as finite matroids. In particular, 
they possess bases ($\subseteq$-maximal independent sets), 
circuits (minimal dependent sets) and a natural notion of duality, in much of the same
way as finite matroids,  see~\cite{infaxioms}. We will use the normal matroid 
terminology. For instance, for any subset $X$ of the ground set $E$
of a matroid $M$ we will write $M|X$ for the restriction of $M$ to $X$, and 
we write $M\mm X=M|(E\sm T)$ for the matroid obtained by deleting the elements in $X$
from $M$.

In~\cite{linkingthm}, the \emph{connectivity function} $\cf$  
is extended to infinite matroids. For any $X\subseteq E(M)$ in a matroid $M$,
choose a basis $B$ of $M|X$ and a basis $B'$ of $M\mm X$, and pick a 
set $F\subseteq B\cup B'$ so that $(B\cup B')\sm F$ is a basis of $M$.
Then we set 
$\cf_M(X):=|F|\in\mathbb N\cup\{\infty\}$ (we do not distinguish between
different infinite cardinalities). 
We remark that the value $\cf_M(X)$ is independent of the choice of the bases
and of the choice of $F$. Moreover, 
$F$ may be chosen to be a subset of $B$ or of $B'$, if necessary.
This definition of the connectivity function has similar properties 
as the traditional connectivity function of a finite matroid. 
For finite matroids, the two notions coincide. For more details and a proof
that $\cf$ is well-defined, see~\cite{linkingthm}.

We call
a partition $(X,Y)$ of $E$ a \emph{$\ell$-separation} if 
$\cf_M(X)\leq \ell-1$ and $|X|,|Y|\geq \ell$.
The matroid $M$ is \emph{$k$-connected} if there exists no $\ell$-separation
with $\ell<k$.

\medskip

Infinite graphs are a natural source of infinite matroids. 
Two dual matroids are normally associated with a finite graph,
the cycle matroid and the bond matroid. These matroids can be extended
verbatim to an infinite graph $G=(V,E)$, 
that we assume to be finitely separable. Let $\I$ be the set
of all edge sets $I\subseteq E$ not containing the edge set
of any finite cycle of $G$. Then $\I$ is the set of independent 
sets of a matroid $\mfc(G)$, the \emph{finite-cycle matroid of $G$}. 
Its circuits are precisely the edge sets of cycles, and its bases
coincide with the \emph{spanning forests}, the sets that form a spanning tree 
on every component. In a similar fashion, we may now define
a matroid whose circuits are the finite bonds, the \emph{finite-bond matroid $\mfb(G)$}. 
However, $\mfc(G)$ and $\mfb(G)$ are no longer dual. 
Rather the dual of $\mfb(G)$ is the \emph{cycle matroid $\mc(G)$}, 
whose circuits are precisely the edge sets of (finite or infinite) cycles of $\itop{G}$.\mymargin{do we need countable here?}
If $G$ is connected then the bases of $\mc(G)$ are the edge sets of topological spanning
trees of $\itop{G}$ and vice versa; see~\cite{matgraphs}.

\medskip

If the graph $G$ is infinite and $2$-connected then the two matroids $\mfc(G)$ and $\mc(G)$ will 
differ. As an illustration, consider again the double ladder in Figure~\ref{dladder}. The 
set of edges in bold will be independent in $\mfc(G)$ but not in $\mc(G)$.

\section{Matroid connectivity in infinite graphs}\label{sec:graphconn}

In a graph $G$, denote for $X\subseteq E(G)$
by $V[X]$ the set of
vertices that are incident with an edge in $X$.
Let $c(X)$
be the number of components of the subgraph~$(V[X],X)$ of $G$.

Our first aim is the following theorem:

\begin{theorem}
\label{mat:matconn}
Let $G$ be a $2$-connected finitely separable graph, and let
$X\subseteq E(G)$, and  $Y:=E(G)\sm X$.
Then the following statements hold:
\begin{enumerate}[\rm (i)]
\item $\cf_{\mfc(G)}(X)=\infty$ if and only $|V[X]\cap V[Y]|=\infty$; and
\item if $\cf_{\mfc(G)}(X)<\infty$ then 
\[
\cf_{\mfc(G)}(X)=|V[X]\cap V[Y]|-c(X)-c(Y)+1.
\]
\end{enumerate}
\end{theorem}
Statement~(ii) is exactly as for finite graphs when the traditional connectivity
function is used, see Tutte~\cite{tutte66}. We shall need two lemmas for the
proof of Theorem~\ref{mat:matconn}.

\begin{lemma}
\label{mat:disjcycles}
Let $G$ be a finitely separable graph, and let $\mathcal D$
be an infinite set of edge-disjoint finite cycles.
Then there exists an infinite subset $\mathcal D'$ of $\mathcal D$
and a vertex $v$ of $G$ so that any two distinct cycles in
$\mathcal D'$ are disjoint outside~$v$.
\end{lemma}
\begin{proof}
Let $C_1,C_2,\ldots$ be an enumeration of (countably many of) the cycles in $\mathcal D$.
Inductively we will delete certain cycles from $\mathcal D$ while ensuring in each step
that we keep infinitely many cycles. In step $i$, assuming $C_i$
has not been deleted, we go through the finitely many vertices
of $C_i$, one by one. Then for a vertex $w$ of $C_i$, 
unless $w$ lies in all but finitely many of 
 the remaining $C_j$, we delete from $\mathcal D$
all those  $C_j$ that contain $w$. If $w$ lies in all but finitely
many of the remaining $C_j$ we skip to the next vertex of $C_i$ without
deleting any cycles. Denote the resulting infinite subset of $\mathcal D'$ by
$\mathcal D$. 

Now, if the cycles in $\mathcal D'$ are pairwise disjoint, choose
any vertex of $G$ for $v$ and observe that this choice of $\mathcal D'$ and $v$
is as desired. So, assume that there is a vertex $v$ shared by two cycles in 
$\mathcal D'$. Pick the smallest index $i$ for which there is a $j\neq i$
so that $C_i$ and $C_j$ both contain~$v$
and so that $C_i,C_j\in\mathcal D'$.
Note that $v$, as well as any other
vertex that lies in two cycles of $\mathcal D'$, is contained in infinitely 
many cycles in $\mathcal D'$; otherwise we would have deleted all but
one of those cycles incident with $v$. 

Suppose there exists a second vertex $w$ contained in two cycles of $\mathcal D'$.
If $k$ is the lowest index with $w\in V(C_k)$ and $C_k\in\mathcal D'$ then why
have we not deleted all those cycles $C_l$ containing $w$ with $l>k$
from $\mathcal D'$ in step $k$? Precisely because all but finitely many
of the cycles in $\mathcal D'$ contain $w$. In particular, infinitely many
of those  cycles in $\mathcal D'$ that contain $v$ must also contain $w$.
By picking a $v$--$w$ path in each of those cycles we obtain infinitely
many edge-disjoint $v$--$w$ paths, which is impossible 
in a finitely separable graph.
\end{proof}

The following lemma is a straightforward combination of Lemmas~4.1 and~4.2
in~\cite{endduality}:
\comment{
\begin{lemma}\selfcite{endduality}
\label{lem:2cols}
Let $G$ be a $2$-connected graph, and let $X$ and $Y$ be two 
sets of edges such that 
$\overline{X}$ and $\overline{Y}$ contain a common end of $G$. Then
there are infinitely many edge-disjoint finite cycles
each of which meets both $X$ and $Y$.
\end{lemma}

\begin{lemma}\selfcite{endduality}
\label{lem:end}
Let $G$ be a $2$-connected finitely separable graph.
If $U$ is an infinite set of vertices then 
$\overline U$ contains an end of $G$.
\end{lemma}

For each vertex in $V[X]\cap V[Y]$ pick one incident edge in $X$
and one in $Y$; denote the set of these edges by $X'\subseteq X$
and $Y'\subseteq Y$, respectively. 
Lemma~\ref{lem:end} yields
an end $\omega\in\overline{X'}$ (where the closure is taken in 
$|G|$). 
It is easy to check that then also
$\omega\in\overline{Y'}$ (again with respect to $|G|$). 

By Lemma~\ref{lem:2cols}, there exists an
infinite set $\mathcal D$ of edge-disjoint finite 
cycles each of which meets both $X'$ and $Y'$. 
}

\begin{lemma}\label{manycycles}
Let $G$ be a $2$-connected finitely separable graph, and let $X',Y'$ be edge sets of $G$
so that there are infinitely many vertices that are incident with both an edge in $X'$ 
and an edge in $Y'$. Then there are infinitely many edge-disjoint finite
cycles in $G$, each of which contains an edge of $X'$ and of $Y'$.
\end{lemma}

We now prove a first part of Theorem~\ref{mat:matconn}.
\begin{lemma}
\label{mat:infint}
Let $G$ be a $2$-connected finitely separable graph, and let
$X\subseteq E(G)$, and  $Y:=E(G)\sm X$. If $|V[X]\cap V[Y]|=\infty$
then $\cf_{\mfc(G)}(X)=\cf_{\mc(G)}(X)=\infty$.
\end{lemma}
\begin{proof}
For each vertex in $V[X]\cap V[Y]$ pick one incident edge in $X$
and one in $Y$; denote the set of these edges by $X'\subseteq X$
and $Y'\subseteq Y$, respectively. Applying Lemma~\ref{manycycles}
in conjunction with Lemma~\ref{mat:disjcycles} 
we obtain 
a vertex $v$ and an infinite set $\mathcal D$ 
of finite cycles, each of which contains an edge of $X$ and of $Y$, 
and so that any two cycles either meet
only in $v$ or not at all. 
As no cycle in $\mathcal D$ has its edge set entirely in $X$ or 
entirely in $Y$ it follows
that neither  $I_X:=X\cap \bigcup_{C\in\mathcal D}E(C)$
nor $I_Y:=Y\cap \bigcup_{C\in\mathcal D}E(C)$ contains the
edge set of a finite 
cycle of $G$. To see that also neither contains
the edge set of an infinite cycle, observe that  
each of the graphs $(V[I_X],I_X)-v$ and $(V[I_Y],I_Y)-v$ 
is the union of (vertex-)disjoint (finite) paths, and therefore
none contains a ray.

Thus $I_X$ and $I_Y$ are independent in both matroids 
$\mfc(G)$ and $\mc(G)$. Let $T_X$
be a basis of $M|X$ containing $I_X$, and let $T_Y\supseteq I_Y$
be a basis of $M|Y$, where $M$ is either $\mfc(G)$ or $\mc(G)$.
Choose $F\subseteq T_X\cup T_Y$ so that $(T_X\cup T_Y)\sm F$ is a basis of $M$. 
Since $I_X\cup I_Y$ contains
the (edge-)disjoint circuits $E(C)$, $C\in\mathcal D$,
$F$ must contain at least one edge from each of those infinitely many circuits. 
Hence $\cf_M(X)=|F|=\infty$.
\end{proof}

\begin{proof}[Proof of Theorem~\ref{mat:matconn}]
(i) By Lemma~\ref{mat:infint} we only need to consider the case when
$|V[X]\cap V[Y]|<\infty$. Pick a basis $T_X$ 
 of $\mfc(G)|X$, and let $T_Y$
be a basis of $\mfc(G)|Y$.
Because $G$ is finitely separable, there is 
a finite set of edges separating $u$ from $v$ in $(V[X],X)$, 
for every of the finitely many pairs
of vertices $u,v\in V[X]\cap V[Y]$. 
Denote by $F$ the union of all those edges, and observe that $F$ is 
a finite edge set. By the choice of $F$ the set $(T_X\cup T_Y)\sm F$
cannot contain any finite circuit, and is thus independent in $ \mfc(G)$.
As $|F|<\infty$ is therefore an upper bound for $\cf_{\mfc(G)}(X)$ the 
result follows.

(ii) 
Pick a spanning tree on every component of $(V[X],X)$ and denote 
the union of their edge sets by $T_X$. We define $T_Y$ for $(V[Y],Y)$
in a similar way. Choose a set of edges $F\subseteq X$ so that $(T_X\cup T_Y)\sm F$
is a basis of $\mfc(G)$, i.e.\ the edge set of a spanning tree of $G$.

We claim that
\begin{equation}
\label{ma:eq:onecomp}
\emtext{if $c(X)=c(Y)=1$ then $\cf_{\mfc(G)}(X,Y)=|V[X]\cap V[Y]|-1$.}
\end{equation}
Let us prove the claim. Each vertex of $U:=V[X]\cap V[Y]$ must lie
in a distinct component of $(V[T_X],T_X\sm F)$ since otherwise there exists a
path in $(V[T_X],T_X\sm F)$ that starts and ends in $U$ but is otherwise disjoint from $U$.
This path can be extended with edges in $T_Y$
to a finite cycle that still misses $F$, which is impossible as $(T_X\cup T_Y)\sm F$ is the edge set of a tree. 
As $(V[T_X],T_X)$ is connected
and as each deletion of a single edge increases the number of components
by exactly one, we obtain $|F|\geq |U|-1$. 
Suppose, on the other hand, that $|F|> |U|-1$. Then there exists a
component  of $(V[T_X],T_X)\sm F$ that contains no vertex of $U$. Pick 
an edge $e\in F$ with one of its endvertices in this component. 
Setting $T:=(T_X\sm F)\cup T_Y$, we observe that 
$\{e\}$ is a cut of $(V[T],T+e)$.
However, as
$T$ is (the edge set of) a spanning tree of $G$, there has to be 
a cycle in $T+e$ containing $e$, a contradiction. This 
proves~\eqref{ma:eq:onecomp}.

We now proceed by induction on $c(X)+c(Y)$, which is indeed a finite
number as $|V[X]\cap V[Y]|$ is an upper bound for both $c(X)$ and $c(Y)$.
Since the induction start is established by~\eqref{ma:eq:onecomp},
we may assume that $(V[X],X)$ has two components $K$ and $K'$. 
Insert a new edge $f$ between $K$ and $K'$, and set $G':=G+f$
and $X':=X\cup\{f\}$. Clearly, $(X',Y)$ is a partition of $E(G')$.
Since $c(X')=c(X)-1$, the induction yields
\[
\cf_{\mfc(G')}(X',Y)
=|V[X]\cap V[Y]|-(c(X)-1)-c(Y)+1.
\]

We shall now show that $\cf_{\mfc(G')}(X',Y)=\cf_{\mfc(G)}(X,Y)+1$. 
Observe that then $T_X+f$ is (the edge set of)
a maximal spanning forest of $(V[X'],X')\subseteq G'$. Moreover,
$(T_X\sm F)\cup T_Y=((T_X+f)\sm (F\cup\{f\}))\cup T_Y$ is a spanning
tree of $G'$, too. Thus
\[
\cf_{\mfc(G')}(X,Y')=|F\cup\{f\}|=|F|+1=\cf_{\mfc(G)}(X,Y)+1,
\]
which finishes the proof.
\end{proof}

Next, let us show that the connectivity functions of $\mfc(G)$
and $\mc(G)$ coincide. For this, we should be able to modify
the proof of Theorem~\ref{mat:matconn} in order to make it work for 
$\mc(G)$, too. Rather then repeating the argument we will
pursue a different approach, for which we will need a small lemma
and a result from~\cite{duality}.

\begin{lemma}
\label{mat:subtop}
Let $G$ be a finitely separable graph, and let $H$ be 
an induced subgraph of $G$ so that $N(G-H)$ is a finite set.
Then every cycle $C\subseteq H$ of $\itop{G}$ contains a cycle of $\itop H$.
\end{lemma}

To prove the lemma, we use a theorem that is the direct consequence
of Theorems~6.3 and~6.5\mymargin{look this up!} of  Diestel and K\"uhn~\cite{tst}: 
\begin{theorem}[Diestel and K\"uhn~\cite{tst}]
\label{cutcriterion}
Let $Z$ be a set of edges in a finitely separable graph $G$.
Then $Z$ is the edge set of an edge-disjoint union of cycles of $\itop G$ 
if and only if $Z$ meets every finite cut of $G$ in an even number of edges.
\end{theorem}

\begin{proof}[Proof of Lemma~\ref{mat:subtop}]
Consider such a cycle $C$ of $\itop G$ that is completely contained in $H$, 
and suppose that $E(C)$ is not the edge set of an edge-disjoint
union of cylces of $\itop H$. 
By Theorem~\ref{cutcriterion} there is a finite cut $F$ of $H$
so that $E(C)\cap F$ is an odd set. The cut $F$ partitions $N(G-H)$
into two sets $A$ and $B$ (one of them possibly empty). Since
every two vertices in $G$ can separated by finitely many edges
there is a finite subset of $E(G)\sm E(H)$ that separates $A$
from $B$ in $G-E(H)$. Choosing a minimal such set $F'$
ensures that $F\cup F'$ is a finite cut of $G$. Then $|E(C)\cap (F\cup F')|=
|E(C)\cap F|$ is odd, implying with Theorem~\ref{cutcriterion} that
$E(C)$ is not the edge set of an edge-disjoint union of cycles of $\itop G$, in
particular that $C$ is not a cycle of $\itop G$, a contradiction. 
\end{proof}

We will make use of the fact that for a connected and finitely separable graph~$G$
there is always a common basis of $\mfc(G)$ and $\mc(G)$: 
\begin{theorem}$\!\!${\bf\cite{duality}}
\label{mat:acircST}
Every connected finitely separable graph $G$  has a spanning
tree that does not contain the edge set of any (infinite) cycle of $\itop G$.
\end{theorem}

\begin{theorem}
\label{mat:sameconn}
Let $G$ be a $2$-connected finitely separable graph.
Then $\cf_{\mfc(G)}(X)=\cf_{\mc(G)}(X)$ for all
$X\subseteq E(G)$.
\end{theorem}
\begin{proof}
Consider a set $X\subseteq E(G)$ and put $Y:=E(G)\sm X$. 
If $V[X]\cap V[Y]$ is an infinite set then 
$\cf_{\mfc(G)}(X)=\cf_{\mc(G)}(X)$
by Lemma~\ref{mat:infint}.

So, assume $V[X]\cap V[Y]$ to be finite. By Theorem~\ref{mat:acircST}
there is for each component $K$ of $(V[X],X)$ a spanning tree not containing
the edge set of any cycle of $\itop K$. 
Lemma~\ref{mat:subtop} ensures
that also no edge set of any cycle of $\itop G$ lies in this spanning tree. 
Consequently, the union $T_X$
of the edge sets of those spanning trees is a basis of $\mfc(G)|X$
as well as of $\mc(G)|X$. We define $T_Y$ analogously for $(V[Y],Y)$.

Next, pick $F\subseteq T_X\cup T_Y$ so that $(T_X\cup T_Y)\sm F$ is a 
basis of $\mc(G)$. Clearly, the set $(T_X\cup T_Y)\sm F$
is independent in $\mfc(G)$, too. If it is even a basis
in $\mfc(G)$ then we have 
$\cf_{\mfc(G)}(X)=|F|=\cf_{\mc(G)}(X)$
as desired. So, suppose $T:=(T_X\cup T_Y)\sm F$ fails to be a 
basis, which implies that $(V[T],T)$ is not (graph-theoretically)
connected. 
As a basis of $\mc(G)$ for a connected graph, $T$ 
is the edge set of a topological spanning tree of $\itop G$.
In particular, the topological spanning tree is path-connected and will 
therefore contain an arc $A$ between two vertices of distinct (graph-theoretical) 
components of $(V[T],T)$. The arc $A$ cannot be a path in the graph, and consequently 
it passes through infinitely many edges. Moreover, it is not hard to check
that if $A$ contains a vertex from $X$ and from $Y$ then it passes through 
the finite vertex set $V[X]\cap V[Y]$. 
Thus there is then also an arc $A'$ between two vertices that has infinitely many edges 
and that is 
completely contained in the closure of $T_X$ or of $T_Y$ (taken in $\itop G$).
We may assume that $E(A')\subseteq T_X$. 
However, as any two vertices in $(V[T_X],T_X)$ are connected
by a finite path as well, such a finite path between the endvertices
of $A'$ plus $A'$ will contain a cycle of $\itop G$ that has all its edges in $T_X$.
This contradicts the definition of $T_X$. Thus, $(V[T],T)$ is connected
and hence $T$ a basis of $\mfc(G)$.
\end{proof}

\section{Proof of main result}

Let us recall the definition of Tutte-connectivity.
A \emph{$\ell$-Tutte-separation} of a graph $G$ is a partition $(X,Y)$ of 
$E(G)$ so that $|X|,|Y|\geq \ell$ and so that $|V[X]\cap V[Y]|\leq \ell$.
We say that a graph $G$ is \emph{$k$-Tutte-connected} 
if $G$ has no $\ell$-Tutte-separation for any $\ell<k$.

The following theorem clearly includes Theorem~\ref{mainthm}:
\begin{theorem}
\label{mat:TutteMat}
Let $G$ be a finitely separable graph. Then for 
integers $k\geq 2$ the following statements are equivalent:
\begin{enumerate}[\rm (i)]
\item $G$ is $k$-Tutte-connected;
\item $\mfc(G)$ is $k$-connected; and
\item $\mc(G)$ is $k$-connected.
\end{enumerate}
\end{theorem}
\begin{proof}
Observe that we may assume $G$ to be  $2$-connected
and that $G$ is an infinite graph. (For finite graphs, 
see Tutte~\cite{tutte66}---note that
$\mfc(G)$ and $\mc(G)$ coincide in this case.)
In light of Theorem~\ref{mat:sameconn} we only need to 
prove that $G$ has a $k$-Tutte-separation with $k\leq m$
if and only if $\mfc(G)$ has an $\ell$-separation
with $\ell\leq m$.

First, let $(X,Y)$ be a $k$-Tutte-separation $(X,Y)$ of $G$,
which implies $|V[X]\cap V[Y]|\leq k$.
Since $c(X),c(Y)\geq 1$ this yields with Theorem~\ref{mat:matconn}
that $\cf_{\mfc(G)}\leq k-1$. Consequently, $(X,Y)$
is a $k$-separation of $\mfc(G)$.

Conversely, let there be an $\ell$-separation in $\mfc(G)$, and choose an
$\ell$-separation $(X,Y)$ of $\mfc(G)$ so that
$c(X)+c(Y)$ is minimal among all $\ell$-separations of $\mfc(G)$.
Since $G$ is infinite, we may assume that $Y$ is an infinite
set.

First, we claim that 
\begin{equation}
\label{Y:conn}
\emtext{$(V[Y],Y)$ is connected}.
\end{equation}
If $(V[Y],Y)$ is not connected then there is a component $K$ of $(V[Y],Y)$
so that $Y':=Y\sm E(K)$ is an infinite set. With $X':=X\cup E(K)$ we
see that both $X'$ and $Y'$ have at least $\ell$ elements. Moreover, it holds that
$|V[X]\cap V[Y]|= |V[X']\cap V[Y']| + |V[X]\cap V[K]|$ and $c(Y)=c(Y')+1$.
The set of  components of $(V[X'],X')$ is comprised of  components of $(V[X],X)$
and of the union of those components of $(V[X],X)$ that have a vertex with $K$ 
in common together with $K$. Since there are at most $|V[X]\cap V[K]|$
components of the latter kind, we obtain $c(X)\leq c(X')+|V[X]\cap V[K]|-1$.
It follows with Theorem~\ref{mat:matconn} that
\begin{eqnarray*}
\cf_{\mfc(G)}(X',Y')& =&|V[X']\cap V[Y']|-c(X')-c(Y')+1\\
&\leq & |V[X]\cap V[Y]| -|V[X]\cap V[K]|-c(X)\\
&&+\,|V[X]\cap V[K]|-1-c(Y)+1+1\\
&=& |V[X]\cap V[Y]|-c(X)-c(Y)+1\leq \ell-1.
\end{eqnarray*}
Thus, $(X',Y')$ is an $\ell$-separation with $c(X')+c(Y')< c(X)+c(Y)$,
contradicting the choice of $(X,Y)$.

Second, we show that
\begin{equation}
\label{Xcomps}
\emtext{$|V[K]\cap V[Y]|\leq \ell$ for every component $K$ of $(V[X],X)$.}
\end{equation}
Suppose there exists a component $M$ of $(V[X],X)$ with $|V[M)]\cap V[Y]|\geq \ell+1$.
Denoting by $\mathcal K$ the components of $(V[X],X)$ we get
\begin{eqnarray*}
\ell-1&\geq & |V[X]\cap V[Y]|-c(X)-c(Y)+1\\
&\geq&\sum_{K\in\mathcal K\sm\{M\}} |V[K]\cap V[Y]| + (\ell+1) -c(X)-c(Y)+1.
\end{eqnarray*}
That $G$ is connected implies $|V[K]\cap V[Y]|\geq 1$ for every $K\in\mathcal K$.
Hence
\[
\ell-1\geq  (c(X)-1) + (\ell+1) -c(X)-c(Y)+1 = \ell+1-c(Y).
\]
This yields $c(Y)\geq 2$, which is impossible by~\eqref{Y:conn}.
Therefore,~\eqref{Xcomps} is proved.

Next, we see that
\begin{equation}
\label{goodcomp}
\emtext{there is a  component $M$ of $(V[X],X)$ with $|E(M)|\geq |V[M]\cap V[Y]|$.}
\end{equation}
If~\eqref{goodcomp} is false then we have $|V[K]\cap V[Y]|\geq |E(K)|+1$ for all $K\in\mathcal K$.
This, however, implies with $c(Y)=1$ that
\begin{eqnarray*}
\ell-1&\geq & |V[X]\cap V[Y]|-c(X)-c(Y)+1\\
&=&\sum_{K\in\mathcal K} |V[K]\cap V[Y]|  -c(X)\\
&\geq&\sum_{K\in\mathcal K} (|E(K)|+1)  -c(X) \,=\, |X|.
\end{eqnarray*}
As $(X,Y)$ is an $\ell$-separation, $X$ is required to have at least $\ell$ elements,
which shows that~\eqref{goodcomp} holds.

Finally, with the component $M$ from~\eqref{goodcomp} we set $\bar X:=E(M)$
and $\bar Y:=E(G)\sm E(M)$. Then $k:=|V[\bar X]\cap V[\bar Y]|=|V[M]\cap V[Y]|\leq \ell$,
by~\eqref{Xcomps}. As $|\bar X|\geq k$ and $|\bar Y|=\infty$ it follows
that $(\bar X,\bar Y)$ is a $k$-Tutte-separation with $k\leq \ell$, as
desired.
\end{proof}
We remark that the arguments in the proof are not new.
Indeed,~\eqref{Y:conn} is inspired by Tutte~\cite{tutte66} and 
steps~\eqref{Xcomps},~\eqref{goodcomp} are quite similar to the proof of 
Lemma~5.3 in~\cite{endduality}.

\section{Tutte-connectivity and duality}

In this final section, we deduce a matroidal proof of the fact that
Tutte-connectivity is invariant under duality (Theorem~\ref{tconnthm}). 

Two finitely separable countable graphs $G$ and $G^*$ 
defined on the same edge set $E$
are a pair 
of \emph{duals} if  any edge set $F\subseteq E$ is the edge set of
a cycle of $\itop G$
if and only if $F$ is a bond of $G^*$.
(A \emph{bond} is a minimal non-empty cut.)
As for finite graphs, a (countable) finitely separable graph is planar
if and only if it has a dual, see~\cite{duality} for a proof and more details.

We need two more results.
\begin{lemma}\selfcite{linkingthm}\label{kapduallem}
The connectivity function is invariant under duality, that is, 
$\cf_M(X)=\cf_{M^*}(X)$ for any subset $X$ of a matroid $M$.
\end{lemma}

\begin{theorem}\selfcite{matgraphs}
\label{graphdual}
Let $G$ and $G^*$ be a pair of countable  dual graphs, each finitely separable, and defined on the same edge set~$E$. Then 
\(
\mc^*(G)=\mfc(G^*).
\)
\end{theorem}

Consider a pair of countable dual graphs $G$ and $G^*$.
Then, by Theorem~\ref{mat:TutteMat},
 $G$ is $k$-Tutte-connected if and only if $\mfc(G)$
is $k$-connected. Since $\mfc(G)=(\mc(G^*))^*$ by Theorem~\ref{graphdual}
and since matroid connectivity is invariant under taking duals 
(Lemma~\ref{kapduallem})
this is precisely the case when $\mc(G^*)$ is $k$-connected. 
Finally, Theorem~\ref{mat:TutteMat} again shows that 
$\mc(G^*)$ is $k$-connected if and only if $G^*$ is $k$-Tutte-connected.
This proves Theorem~\ref{tconnthm}.

\bibliographystyle{amsplain}
\bibliography{../graphs}

\providecommand{\bysame}{\leavevmode\hbox to3em{\hrulefill}\thinspace}
\providecommand{\MR}{\relax\ifhmode\unskip\space\fi MR }
\providecommand{\MRhref}[2]{%
  \href{http://www.ams.org/mathscinet-getitem?mr=#1}{#2}
}
\providecommand{\href}[2]{#2}
\begin{thebibliography}{10}

\bibitem{thindual}
H.~Afzali and N.~Bowler, \emph{Thin sums matroids and duality}, Preprint 2012.

\bibitem{matunion}
E.~Aigner-Horev, J.~Carmesin, and J.~Fr{\"o}hlich, \emph{Infinite matroid
  union}, Preprint 2012.

\bibitem{2sepmat}
E.~Aigner-Horev, R.~Diestel, and L.~Postle, \emph{The structure of
  2-separations of infinite matroids}, Preprint 2012.

\bibitem{matintbase}
N.~Bowler and J.~Carmesin, \emph{Matroid intersection, base packing and base
  covering for infinite matroids}, Preprint 2012.

\bibitem{infintmat}
\bysame, \emph{Matroids with an infinite circuit-cocircuit intersection},
  Preprint 2012.

\bibitem{duality}
H.~Bruhn and R.~Diestel, \emph{Duality in infinite graphs}, Comb.,\ Probab.\
  Comput. \textbf{15} (2006), 75--90.

\bibitem{matgraphs}
\bysame, \emph{Infinite matroids in graphs}, Disc.\ Math. \textbf{311} (2011),
  1461--1471.

\bibitem{endduality}
H.~Bruhn and M.~Stein, \emph{Duality of ends}, Comb.,\ Probab.\ Comput.
  \textbf{19} (2010), 47--60.

\bibitem{linkingthm}
H.~Bruhn and P.~Wollan, \emph{Finite connectivity in infinite matroids},
  Europ.\ J.\ Comb. \textbf{33} (2012), 1900--1912.

\bibitem{diestelBook10}
R.~Diestel, \emph{Graph theory \emph{(4th edition)}}, Springer-Verlag, 2010.

\bibitem{infaxioms}
R.~Diestel, H.~Bruhn, M.~Kriesell, R.~Pendavingh, and P.~Wollan, \emph{Axioms
  for infinite matroids}, Preprint 2010.

\bibitem{tst}
R.~Diestel and D.~K{\"u}hn, \emph{Topological paths, cycles and spanning trees
  in infinite graphs}, Europ.\ J.\ Combinatorics \textbf{25} (2004), 835--862.

\bibitem{Higgs69c}
D.A. Higgs, \emph{Infinite graphs and matroids}, Recent Prog.\ Comb., Proc.\
  3rd Waterloo Conf., 1969, pp.~245--253.

\bibitem{Higgs69}
\bysame, \emph{Matroids and duality}, Colloq.\ Math. \textbf{20} (1969),
  215--220.

\bibitem{Oxley92}
J.G. Oxley, \emph{Infinite matroids}, Matroid applications (N.~White, ed.),
  Encycl.\ Math.\ Appl., vol.~40, Cambridge University Press, 1992, pp.~73--90.

\bibitem{Moritz}
M.~Schulz, \emph{Der {Z}yklenraum nicht lokal-endlicher {G}raphen}, Diploma
  thesis, Universit{\"a}t {H}amburg (2005), \textbf{\small\tt
  http://www.math.uni-hamburg.de/home/diestel/papers/others/Schulz.{Diplomarbe%
it.pdf}}.

\bibitem{tutte66}
W.T. Tutte, \emph{Connectivity in matroids}, Can.\ J.\ Math. \textbf{18}
  (1966), 1301--1324.

\end{thebibliography}

\small
\vskip2mm plus 1fill
\parindent=0pt

$\,$\vfill
Version 23 Oct 2012
\bigbreak

Henning Bruhn
{\tt <bruhn@math.jussieu.fr>}\\
\'Equipe Combinatoire et Optimisation\\
Universit\'e Pierre et Marie Curie\\
4 place Jussieu\\
75252 Paris cedex 05\\
France

\end{document}